\documentclass[a4paper,12pt]{amsart}
\usepackage{fullpage}
\usepackage{amsmath}
\usepackage{amsfonts}
\usepackage{amssymb}
\usepackage{color}
\usepackage[OT4]{fontenc}
\usepackage[latin2]{inputenc}
\usepackage{graphicx}
\usepackage{epstopdf}
\usepackage{caption}
\usepackage{subcaption}
\usepackage{adjustbox}
\usepackage{hyperref}

\DeclareMathOperator\C{{\it\mathbb C^n}}

\DeclareMathOperator\R{{\it \mathbb R}}

\DeclareMathOperator\Om{{\it\Omega}}

\DeclareMathOperator \lbr {{\it\lbrace}}
\DeclareMathOperator \rbr {{\it\rbrace}}

\DeclareMathOperator\pa{\partial}

\def\Ga{\Gamma}
\def\Om{\Omega}

\newtheorem{theorem}{Theorem}

\newtheorem{lemma}[theorem]{Lemma}

\newtheorem{corollary}[theorem]{Corollary}

\newtheorem{example}[theorem]{Example}
\newtheorem{remark}[theorem]{Remark}
\title{Alexandrov-Bakelman-Pucci type estimate for plurisubharmonic functions}
\author{S\l awomir Dinew}\author{\.Zywomir Dinew}
\address{S\l awomir Dinew\\
Department of Mathematics and Computer Science\\
Jagiellonian University, Poland}\email{slawomir.dinew@im.uj.edu.pl}
\address{\.Zywomir Dinew\\
	Department of Mathematics and Computer Science\\
	Jagiellonian University, Poland}\email{zywomir.dinew@uj.edu.pl}

\begin{document}
\begin{abstract}We prove an optimal Alexandrov-Bakelman-Pucci type estimate for plurisubharmonic functions  without assuming their continuity. This generalizes a result of Y. Wang. As a corollary we  generalize an estimate from \cite{DD18}. We also address a problem posed in \cite{Wan12}.
\end{abstract}
\maketitle

\section{Introduction}

The Alexandrov weak maximum principle is a basic tool in modern PDE theory. In its classical version for a function $u\in C^{2}(\Omega)\cap C(\bar{\Omega})$ living in a bounded domain
$\Omega$ it reads:

\begin{equation}\label{alexandrov} \sup_{\Omega} u\leq \sup_{\partial\Omega} u+\frac{diam(\Omega)}{\omega_{n}^{1/n}}\left(\int _{\{-u=\Gamma_{-u}\}}|\det D^2 u|\right)^{\frac{1}{n}},\end{equation}
see Lemma 9.2 in \cite{GT01}. Recall that $\omega_n$ above stands for  the volume of the unit ball in $\R^n$ and $\{-u=\Gamma_{-u}\}$ is the so-called contact set (see below). 
This inequality is especially fundamental in the viscosity theory of nonlinear elliptic second order   equations (see \cite{CIL92}
for the notions of viscosity theory). In particular it is instrumental in the proof of the more general Alexandrov-Bakelman-Pucci estimate, which establishes a uniform bound
on the viscosity supersolutions $u$ of the
equation
\begin{equation}\label{secord}F(D^2u)=f,\end{equation}
with $F$ being a uniformly elliptic second order differential operator and $f\in C(\Omega)$. The Alexandrov-Bakelman-Pucci estimate, or ABP for short, reads:

\begin{equation}
\label{ABP}
\sup _{\Omega} u^{-}\leq C diam(\Omega)\left(\int_{\Omega\cap\{u=\Gamma_u\}} (f^{+})^n\right)^{\frac{1}{n}},\end{equation}
(see  Theorem 3.6 in \cite{CC95}), where $u$, which is continuous and non-negative on the boundary, is a viscosity supersolution of \eqref{secord},  $u^-$ denotes $\max\lbr -u,0\rbr$, 
$f^{+}:=\max\lbr f,0\rbr$ and
 $C$ is a universal constant. 

Nowadays many improvements  of this estimate exist under special assumptions on $u$ or on the equation (see \cite{Cab95}, \cite{CCKS96}, \cite{AIM06} to mention just a few). 
In any case a Sobolev regularity of order at least $W^{2,p}_{loc}, p>\frac{n}{2}$ is required for $u$ for the theory to work for general second order nonlinear equations. Note that by the Sobolev embedding this 
forces $u$ to be continuous. 

Recently viscosity methods  were applied for the complex Monge-Amp\`ere equation, see \cite{Zer13}. 
 In \cite{Wan12} Y. Wang, extending results from \cite{Blo05} and \cite{CP92}, proved in this setting  that  if  $u$ is plurisubharmonic, ($PSH$ for short) {\it and continuous} 
  one has the following bound:
 
 \begin{equation}
 \label{wang12}
 \sup_{\Omega}u^{-}\leq Cd\Vert f\cdot\chi_{\{u=\Gamma_u\}}\Vert_{L^2(\Omega)}^{\frac{1}{n}},\end{equation}
 where $\Omega\subset\subset B_d$ ($B_d$ is a ball of radius $d$), $u\in C(\bar\Omega)$ satisfies $(dd^c u)^n\leq f$ in the viscosity sense, $u_{|\partial\Omega}\geq 0$, and $f\geq 0$ is a real valued
 function from $C(\bar\Omega)$. Note that in this bound no convexity assumptions on $\Omega$ are made and, more importantly, the $L^2$ norm on the right
 hand side is taken only over a special subset of $\Omega$.
 
 The standard definitions in viscosity theory require that viscosity supersolutions have to be {\it lower semicontinuous} - \cite{CIL92, Zer13}.
 On the other hand plurisubharmonic functions are axiomatically {\it upper semicontinuous}. Hence the continuity assumption in Wang's result is
 natural from the viscosity point of view.
 
 On the other hand there are many $PSH$ functions $u$ which fail to be continuous, yet the Monge-Amp\`{e}re operator
 $(dd^c u)^n$ is well-defined in the sense of pluripotential theory. In fact Bedford and Taylor defined $(dd^c u)^n$ as a non-negative 
 Borel measure for a continuous plurisubharmonic function $u$ in \cite{BT76}, and then generalized the construction to
 $u\in L^{\infty}_{loc}$ in \cite{BT82}. We recall that this passage is not just a matter of technicalities. It requires delicate potential theoretic
 arguments,  but the construction allowed the resolution of several long-standing open problems (see \cite{BT82} for more details). Later on B\l ocki \cite{Blo04},\cite{Blo06} found the exact conditions on $u$ under
 which Bedford and Taylor's definition  can be applied. In fact many discontinuous $PSH$ functions have measures with smooth densities - any discontinuous $PSH$ function dependent
 on a fewer than $n$ variables would do. There are also other types of {\it maximal} $PSH$ functions which are discontinuous (see for example \cite{Sic81}).
 We shall also provide such examples with almost everywhere positive densities (see  Example \ref{example} below). 
 
 This clearly shows that there is a discrepancy between pluripotential and viscosity {\it supersolutions} - a fact that has been observed already in
 \cite{EGZ11}. On the bright side pluripotential and viscosity {\it subsolutions} are equivalent (see \cite{EGZ11, Zer13}). We refer the reader to the recent paper
 \cite{GLZ}, where inequalities for mixed Monge-Amp\`ere measures are studied from viscosity and pluripotential viewpoint. It is worth pointing out that in \cite{GLZ} the lack of continuity
 is a serious source of troubles (inequalities for mixed Monge-Amp\`ere measures of continuous $PSH$ functions can be studied using much simpler tools, see \cite{Blo96}).
 
 Despite these discrepancies there are also results linking both theories. In fact an easy argument (see \cite{Zer13})
 shows that a
 pluripotential supersolution $u$ with continuous right-hand side becomes viscosity supersolution once the {\it lower semicontinuous regularization}
 $$u_{*}(w):=\liminf_{z\rightarrow w}u(z)$$
 is applied (we use the convention ${\displaystyle \liminf_{z\rightarrow w}=\liminf_{\stackrel{z\rightarrow w}{z\neq w}}}$). We note that for continuous up to the boundary $u$ this regularization keeps $u$ fixed and hence continuous pluripotential supersolutions {\it are} also
 viscosity supersolutions. We also note that for a generic lower semicontinuous function $u$ it holds that $u_{*}\geq u$, but in general $u_{*}\neq u$ (even more: $(u_{*})_{*}$ need not be equal to $u_{*}$), whereas if $u$ is upper semicontinuous, in particular plurisubharmonic, then $u_{*}\leq u$.
 
 Yet another subtle issue is the continuity up to the boundary and the {\it right notion} of boundary values. The standard assumption $u\in C(\overline{\Omega})$
 resolves all these issues in the continuous setting. It is thus worth pointing out that Wang's estimate {\it fails} dramatically if one merely assumes
 $u\in C(\Omega)$ and $u$ is defined on $\pa\Omega$ by ${\displaystyle u(z):=\limsup_{\Omega\ni w\rightarrow z\in\pa\Omega}u(w)}$, which is the standard potential-theoretic extension making $u$ upper semicontinuous on $\bar\Omega$, (see Example \ref{ex}). Discarding the boundary continuity assumption has further negative consequences. For example one can no longer use the uniqueness of solutions to the Dirichlet problem or various versions of the comparison principle.   
 
  The aim of this paper is to investigate  whether one can relax the continuity assumptions
 in Wang's argument (with suitable modifications) and prove an Alexandrov-Bakelman-Pucci type estimate in the special case of bounded plurisubharmonic
 $u$ and right hand side function $f\in L^{2}(\Omega)$. The affirmative answer is summarized 
 in the following main theorem:
 
 \begin{theorem}\label{ABPestimate}
 	Let $u$ be a bounded plurisubharmonic function (not necessarily continuous) in a bounded domain $\Om\subseteq\mathbb C^{n}$. Suppose that $(dd^cu)^n\leq f$ as measures for some non-negative function  $f\in L^{2}(\Om)$. 
 	Then the following Alexandrov-Bakelman-Pucci type inequality holds:
 	\begin{equation}
 	\label{mainn}
  \sup_{\Omega}u^{-}\leq \sup_{\pa \Omega} (u_{*})^{-}+Cdiam(\Omega)\Vert f\cdot\chi_{\{ \Gamma_{u_{*}}=\widetilde{u_{*}}\}}\Vert_{L^2(\Omega)}^{\frac{1}{n}},	\end{equation}
 	where $C$ is a numerical constant dependent only on the dimension $n$. 
 \end{theorem}
 \begin{remark}
 If $u$ is plurisubharmonic and defined in a larger domain $U$ containing $\bar\Omega$, then one can use ${\displaystyle \liminf_{\Omega\ni w\rightarrow z\in\pa\Om}u}$ rather than $u_*$ on the boundary of 
 $\Omega$ which results in a slightly better bound in the above inequality - see Example \ref{ex3}.
 \end{remark}

Here and below whenever the measure $(dd^c u)^n$ is absolutely
 continuous with respect to the Lebesgue measure we will write, abusing the notation slightly,  $(dd^c u)^n=f$, where  $f$ is the density of the measure.
 
 As one application of this generalization we mention that
 the following theorem was proved in \cite{DD18} (Theorem $31$) with the extra assumption, that $u$ is continuous.
 \begin{theorem}\label{blocki} Let $U\subseteq \C$ be a domain that contains the ball $$B_R(z_{0})=\{z\in\C:  \Vert z-z_{0}\Vert\leq R\}.$$ Assume that a continuous $u\in PSH(U)$ obeys the conditions:
 	\begin{enumerate}
 		\item For some $\Lambda>0$  it holds that $\Lambda d\lambda^{2n}\geq (dd^cu)^{n}\geq 0$ on $B_R(z_{0})$ as measures, where $\lambda^{2n}$ is the Lebesgue measure.
 		\item There exists $\sigma>0$ such that  $u(z)\geq \sigma ||z-z_{0}||^2$ when $z\in B_R(z_{0})$ and $u(z_{0})=0$.
 	\end{enumerate}
 	Then, there exists a constant $c=c(n,\Lambda)$ such that $u(z)\leq \frac{c}{\sigma^{2n-1}}\Vert z-z_{0}\Vert^{2}$ for all 
 	$z\in B_R(z_{0})$.
 \end{theorem}
 
 The only place where we needed the continuity was Lemma $29$ in \cite{DD18}, which  now  can be substituted
 by Corollary \ref{pota} below, so the continuity assumption can  be dropped.
 
 \begin{theorem}\label{three}
 	The conclusion of Theorem \ref{blocki} holds even without the continuity assumption on $u$. 
 \end{theorem}
 
 In \cite{Wan12}, Wang   posed the following problem  (see Remark 12 there):  Ko\l odziej's estimate yields that if $0\leq f\in L^p(\Omega)$, for some $p > 1$, then the plurisubharmonic solution $u$ of $(dd^cu)^n =f$, $u\geq0$ on $\partial\Omega$ satisfies
 $$\sup_{\Omega} u^{-}\leq C(p, n, diam(\Omega))\| f \|_{L^{p}(\Omega)}^{\frac1n}.$$ Comparing this with \eqref{wang12} or \eqref{mainn} one wonders whether
 or not $\| f\chi_{\{\Gamma_{ u_{*}}=\widetilde{u_{*}}\}}\|_{L^2(\Omega)}$ can control
 $\| f \|_{L^{p}(\Omega)}$ or vice versa. We show that the answer is negative in general, see Example \ref{wanggg} below.  Ko\l odziej's estimate itself will be treated in a subsequent paper.
 
 We also present some  examples, further remarks,  and applications of Theorem \ref{ABPestimate}.
  \section{Proof of the main theorem}
  We refer to \cite{Kol05,Blo96} for the basics of pluripotential theory, in particular for the construction of the Monge-Amp\`ere measures for locally bounded plurisubharmonic functions.
  For the viscosity theory good sources are \cite{CIL92,CC95} and for the special case of the viscosity theory of the complex Monge-Amp\`ere operator we refer to the survey \cite{Zer13}.
  
  Recall that for a convex function $v$ defined on a domain $\Omega$ (treated as a subdomain of $\mathbb C^n$ identified with $\mathbb R^{2n}$) the {\it gradient image} is defined as follows:
  \begin{equation}\label{gradim}
   \partial v(x):=\lbrace p\in \mathbb R^{2n}:  \  v(y)\geq v(x)+\langle p,y-x\rangle,\ \forall y\in \Omega\rbrace,
  \end{equation}
  with $\langle p,q\rangle$ denoting the usual Euclidean inner product. Note that, by convexity, it does not matter whether the inequality holds in the whole $\Omega$ or just locally around $x$, that is, the definition of the gradient image is independent of $\Omega$.   More generally for a Borel set $A$ the gradient image of $A$
  is defined by 
  $$\partial v(A):=\bigcup_{x\in A}\partial v(x).$$

  It is a classical fact (see Lemma 1.1.12 \cite{Gu01} for instance) that for almost every vector $p\in\partial v(A)$ there is a 
  {\it unique} $x\in A$ such that $p$ is in the gradient image of the point $x$. This fact leads to the classical construction of
  Alexandrov's Monge-Amp\`ere measure of a convex function (see Section 1.1 in \cite{Gu01} for a modern exposition):
  \begin{theorem} Given a convex function $v$ on a domain $\Omega$ and any Borel subset $A\subseteq\Omega$ the set function
  $$Mv(A):=\lambda^{2n}(\partial v(A))$$
   is a Borel measure, which is finite on compact sets. 
  \end{theorem}
  
  The real Monge-Amp\`ere measure can also be defined, still for a convex function $v$, through analytic methods - see \cite{RT77}. A simple but
  fundamental observation - Proposition 3.4 in \cite{RT77}, states that both constructions are in fact equivalent:
  \begin{theorem}\label{RT}
  Let $v$ be a convex function defined in a domain $\Omega$. Then the Alexandrov and the weak Monge-Amp\`ere measures of $v$ agree on Borel subsets of $\Omega$. 
  \end{theorem}

  Let $U$ be a fixed bounded  domain and $u$ be a lower semicontinuous real valued function on $U$, which is bounded below on $U$, and such that \begin{equation}\label{compat}\displaystyle{\liminf_{ U\ni z\to w}u(z)\geq 0}, \text{ for any } w\in\partial U.\end{equation} Fix a ball $B_d$ of radius $d$ such that $U\subset\subset B_{d}$ and let $B_{2d}$ be a concentric ball of radius $2d$.
 Denote by $\Gamma_u$ the \emph{ convex envelope} of $u$ defined as follows:   We extend $\min\{u,0\}=-u^{-}$ by zero from $U$ to $B_{2d}$   and call this extension $\tilde{u}$. Also
 \begin{equation}
 \label{envelope}
 \Ga_u(x)=\Ga_{u,B_{d}}(x):=\sup\lbr l(x):\ l \quad{\rm is\quad affine}, l\leq \tilde{u}\ {\rm in }\ B_{2d}\rbr, x\in B_{2d},\end{equation}
 and $C_u=C_{u,B_{d}}:=\lbr\Ga_u=\tilde u\rbr$ is the so-called \emph{ contact set} of $u$. Note that in \cite{Wan12} there is a typo
 in Definition $4$, seeming to imply that $l\leq \tilde{u}\ {\rm only\ in }\ U$, not in $B_{2d}$.    Unless $\tilde u=0$,
 we have  $C_{u}\subseteq U$ and $C_{u}=\{\Ga_u= u\}$ and we will assume this from now on. 
 
  Usually some extra assumptions such as continuity (\cite{GT01},\cite{CC95}) are made on $u$, but just lower
 semicontinuity  is needed to ensure that the contact set is closed. Note  that condition \eqref{compat} guarantees that $\tilde u$ is lower semicontinuous, whenever $u$ is.  The function $\Ga_u$ is convex, hence continuous, and the supremum in \eqref{envelope} is attained
 at every point, since graphs of convex functions allow supporting hyperplanes at every point. Even if $u$ is convex in
 $U$ then $u\neq\Ga_u$ and $C_{u}\neq U$, unless $u\equiv0$. 
 
  For lower semicontinuous functions $u$ such that ${\displaystyle\liminf_{U\ni z\to w}u(z)}$ is negative for some $w\in\partial U$, we first extend $u$ as a lower semicontinuous function on $\bar U$, which we also denote by $u$. This is done  by setting ${\displaystyle u(w)=\liminf_{U\ni z\to w}u(z)}$ for $w\in \partial U$. Next we define $\Gamma_{u}$ and $C_{u}$ as ${\displaystyle\Gamma_{u+\sup_{\pa U} u_{}^-}}$ and ${\displaystyle\{ \Gamma_{u+\sup_{\pa U} u_{}^-}=\widetilde{u+\sup_{\pa U} u_{}^-}\}}$ respectively. Note that the estimate we want to prove is not completely invariant with respect to adding constants to $u$, since the contact set may change, but choosing so will give us the sharpest form of the estimate.
  
   Clearly $u+\sup_{\pa U}{u}_{}^-$ satisfies the condition \eqref{compat}. Of course 
  $(u+\sup_{\pa U}{u}^-)^-\neq u^--\sup_{\pa U}{u}^-$, but \begin{equation}\label{supp}\sup_{U}(u+\sup_{\pa U}{u}^-)^-
  =\sup_{U}( u^--\sup_{\pa U}{u}^-)= \sup_{U}u^--\sup_{\pa U}{u}^-.\end{equation}
   
   If $u\in PSH(U)$ then $u_{*}$ is lower semicontinuous on $\bar U$ and $\sup_{U} u^{-}=\sup_{U} u_{*}^{-}$. This is not true for $\sup_{\pa U} u^{-}$ and $\sup_{\pa U} u_{*}^{-}$, as shown by Example \ref{ex}.
   
  The following two lemmas are well-known to the experts - see Lemma 1.4.4 in \cite{Gu01}, where a continuous version is proven. We include a sketch for the sake of completeness:
  \begin{lemma}\label{gutierrez}
  Let $v$ be a lower semicontinuous function on the closure of a bounded domain $\Omega$ contained in a ball $B_{d}$. Let also  $v\geq 0$ on $\partial{\Omega}$ while $v(x_0)<0$ for some $x_0\in\Omega$. Define
  $$V(x_0):=\lbrace q\in\mathbb R^{2n}: v(x_0)+\langle q,\xi-x_0\rangle<0,\ \forall \xi\in\overline{B}_{2d}\rbrace.$$
  Then
  $$V(x_0)\subseteq \partial \Gamma_v(\lbrace  \Gamma_v=\tilde{v}\rbrace).$$

  \end{lemma}
  
  \begin{proof}
   Assume that the vector $q$ belongs to $V(x_0)$. Note that the supremum \newline
   $\lambda_0:=\sup\lbrace \lambda: \lambda+\langle q,\xi-x_0\rangle\leq \tilde{v}(\xi), \forall \xi\in\overline{B}_{2d}\rbrace$
   is attained as $\tilde{v}(\xi)-\langle q,\xi-x_0\rangle$ is lower semicontinuous. Then $\lambda_0\leq v(x_0)<0$ as the evaluation at $x_0$ shows.
   Furthermore still by the lower semicontinuity of $\tilde{v}$ there exists a point $\hat{\xi}\in\overline{B}_{2d}$, such that
   $\tilde{v}(\hat{\xi})=\lambda_0+\langle q,\hat{\xi}-x_0\rangle$. As $q\in V(x_0)$ we have that 
   $\tilde{v}(\hat{\xi})<0$ and $\tilde{v}=0$ on $\overline{B}_{2d}\setminus\Omega$ now implies that $\hat{\xi}\in \Omega$. But then
   $\hat{\xi}\in\lbrace \Gamma_v=\tilde{v}\rbrace$ and finally $q\in\partial \Gamma_v(\lbrace \Gamma_v=\tilde{v}\rbrace)$, as claimed.
  \end{proof}
  
  The lemma implies that $V(x_0)\subseteq \pa\Gamma_v(\lbrace\Gamma_{v}=\tilde{v}\rbrace)$. On the other hand it is easy to see that
  the ball $B_{\frac{-v(x_0)}{2d}}\left(0 \right)$ is contained in $V(x_0)$, hence
  
  \begin{equation}\label{Gut2}
\omega_{2n}\frac{(-v(x_0))^{2n}}{(2d)^{2n}} \leq \lambda^{2n}\left(V(x_0)\right)\leq \lambda^{2n}\left(\pa\Gamma_v(\lbrace\Gamma_{v}=\tilde{v}\rbrace)\right).
  \end{equation}

 As a corollary we obtain the following weak Alexandrov maximum principle (compare with Theorem 1.4.5 in \cite{Gu01}):
 \begin{lemma}\label{145}
Let $v$ be a lower semicontinuous function on the closure of a bounded domain $\Omega$. Then
$$\sup_{\Omega}(v^{-})\leq \sup_{\partial \Omega}(v^{-})+C\ diam(\Omega)\ \lambda^{2n}\left(\pa\Gamma_v(\lbrace\Gamma_{v}=\tilde{v}\rbrace)\right)^{\frac1{2n}}.$$
 \end{lemma}
As a result, the Alexandrov-Bakelman-Pucci estimate boils down to establishing a bound on the volume of the gradient image of the contact set. In \cite{Wan12}
this is done by exploiting the fact that for continuous plurisubharmonic $u$ and continuous right hand side $f$, the function
$\Gamma_u$ is a viscosity supersolution to
$$(dd^cu)^n=f^2\chi_{\lbrace \Gamma_{u}=\tilde{u}\rbrace}.$$
 
  In the viscosity approach one would look for a  lower differential tests at points of the contact set.
In our setting no viscosity tools are available since the right hand side is merely
measurable.

Instead we shall construct a different function in the following crucial lemma:
\begin{lemma}\label{convexf}
 Let $u$ be a bounded plurisubharmonic function in a domain $\Omega$ such that $(dd^cu)^n\leq f$ as measures for some $f\in L^2(\Omega)$. Let then
 $\Gamma_{u_*}$ be the convex envelope of $\widetilde{\, u_{*}}$ and  $\displaystyle{z_0\in \lbrace \Gamma_{u_*}=\widetilde{\, u_{*}}\rbrace}$. Fix small positive $r<dist(z_{0},\partial\Omega)$. Let finally the {\it convex}
 function $v$ solve the real Monge-Amp\`ere Dirichlet problem
 \begin{equation}\label{realma}
  \begin{cases}
   v\in C(\overline{B_r(z_0)});\\
   \det D^2v=\frac{f^2}{4^n(n!)^2};\\
   v|_{\partial B_r(z_0)}=\Gamma_{u_*}.
  \end{cases}
 \end{equation}
Then $v\leq \Gamma_{u_*}$ in $B_r(z_0)$.
\end{lemma}
\begin{proof} Note that $v$ need not agree with $u_{*}$ at $z_0$.  
 Observe that if $v$ was additionally smooth
 then $(dd^cv)^n=4^nn!\det(v_{i\bar{j}})\geq2^nn!\sqrt{\det D^2 v}\geq f$ from a comparison result of the real and complex Hessians  of  a convex function - see \cite{Blo05}.
 But a possibly singular $v$ is locally a uniform limit of smooth convex approximants $v_j$ (standard convolutions with smoothing kernel would do), and passing to the limit
 we obtain $(dd^cv)^n\geq f$ as measures for {\it any} such convex solution $v$.
 
 Next, $\Gamma_{u_*}\leq u_*\leq u$ together with $(dd^cu)^n\leq f$ gives that $v\leq u$ in $B_r(z_0)$, by the comparison principle for plurisubharmonic functions (see \cite{BT82} or \cite{Kol05}). But now $v$ is continuous, hence $v\leq u_*$. Note also that
 $v|_{\partial B_r(z_0)}\leq 0$, hence $v$  is non-positive in the interior of the ball. Thus $v\leq \widetilde{\, u_*}$ and  finally
 $v\leq \Gamma_{u_*}$.
\end{proof}

The main theorem now follows in the following way: As $v\leq \Gamma_{u_*}$ on $B_r(z_0)$ with equality on the boundary, 
$\partial \Gamma_{u_*}(B_r(z_0))\subseteq\partial v(B_r(z_0))$, by Lemma 1.4.1 in \cite{Gu01}. Hence
$$\lambda^{2n}\left(\partial \Gamma_{u_*}(B_r(z_0))\right)\leq \lambda^{2n}\left(\partial v(B_r(z_0))\right)=\int_{B_r(z_0)}\frac{f^2}{4^n(n!)^2}, $$
where we used Theorem \ref{RT} to justify the last equality. 

In particular this means that the Alexandrov measure of $\Gamma_{u_*}$ restricted to the contact set is majorized by $\frac{f^2}{4^n(n!)^2}$.
As a result 
$$\lambda^{2n}\left(\pa \Gamma_{u_*}(\lbrace \Gamma_{u_*}=\widetilde{\, u_*}\rbrace)\right)^{\frac1{2n}}\leq\frac{1}{2\sqrt[n]{n!}}\left(\int_{\lbrace \Gamma_{u_*}=\widetilde{u_*}\rbrace}f^2\right)^{\frac 1{2n}},$$
and coupling this with Lemma \ref{145} applied for $v=u_*$ the main result follows.

\section{Applications and remarks}
\begin{remark}
	The a priori assumption of boundedness on $u$ can not be dropped, since it is not true  that $f\in L^2(\Omega)$ yields $u\in L^{\infty}$, as the example of a pluricomplex Green function $u$ on $\Omega_{1}$ shows, where $\Omega=\Omega_{1}\setminus \{w\}$, with $w$ being the pole of $u$. For unbounded $u$, the notion of convex contact set is no longer meaningful, at least if one keeps the standard definition.   
\end{remark}
\begin{remark}\label{optimalconstant}
	Following our proof carefully, we get that the constant in \eqref{mainn} can be taken as $C=\frac{1}{2\sqrt{\pi}\sqrt[2n]{n!}}$. One can not get a smaller constant as the following example shows. 	
	 Take $u(z)$ such that $u$ is plurisubharmonic, $u\geq 0$ on $\partial\Omega$, where $\Omega=B_{d}(0)$, $u(z)>\frac{a}{\sqrt{4d^2-a^2}}(\|z\|-2d)$ on $B_{d}(0)\setminus\overline{B_{\frac{a^2}{2d}}(0)}$, and  $u(z)=-\sqrt{a^2-\|z\|^2}$ on a neighborhood of $\overline{B_{\frac{a^2}{2d}}(0)}$, for some $0<a\leq d$.  The contact set is $B_{\frac{a^2}{2d}}(0)$ and   $f=(dd^cu)^n=4^nn!\det(u_{i\bar j})=4^nn!\frac{1}{2^{n+1}}\frac{2a^2-\|z\|^2}{(a^2-\|z\|^2)^{\frac{n+2}{2}}}=2^{n-1}n!\frac{2a^2-\|z\|^2}{(a^2-\|z\|^2)^{\frac{n+2}{2}}}$ there. The integral of $f^2$ over the contact set is a complicated expression, fortunately the real Hessian  of $u$, which is $\det D^2 u=\frac{a^2}{(a^2-\|z\|^2)^{n+1}}$, is both comparable and easily explicitly integrable  there. We have
	 $$1\leq \frac{f^2}{4^n(n!)^2\det D^2 u }\leq 1+\frac{a^4}{64d^4-16a^2d^2}$$
	 on the contact set, so \eqref{mainn}  yields
	  
	  $$a\leq 2d C\sqrt[2n]{\int_{B_{\frac{a^2}{2d}}(0)}f^2}\leq2d C\sqrt[2n]{\left(1+\frac{a^4}{64d^4-16a^2d^2}\right)\int_{B_{\frac{a^2}{2d}}(0)}4^n(n!)^2\det D^2 u}$$$$=2dC\sqrt[2n]{\left(1+\frac{a^4}{64d^4-16a^2d^2}\right)\int_{B_{\frac{a^2}{2d}}(0)}4^n(n!)^2 \frac{a^2}{(a^2-\|z\|^2)^{n+1}}}$$$$=2d C\sqrt[2n]{\left(1+\frac{a^4}{64d^4-16a^2d^2}\right)} 2\sqrt[2n]{n!}\sqrt{\pi}\frac{a}{\sqrt{4d^2-a^2}}.$$
	  Now letting  $a\to0^+$ proves  the claim.

	Pursuing the task of obtaining the best constant possible,  we can modify our construction by assuming 
	$\Omega \subseteq B_d\subseteq B_{d+\varepsilon}$ and taking the convex envelope with respect to $B_{d+\varepsilon}$ instead of $B_{2d}$ (the definition of contact sets changes accordingly). With a few modifications of the proof we get an ABP estimate with the constant $C=\frac{d+\varepsilon}{4d\sqrt{\pi}\sqrt[2n]{n!}}$ and the same example as above shows that it is optimal.  In the limit when $\varepsilon\to0^+$ we get a slightly better constant than would directly correspond to \eqref{alexandrov}.  
\end{remark}
\begin{remark}
	The same example demonstrates that it is not possible to obtain the ABP estimate with optimal constant while integrating over a set which is essentially smaller than the contact set.
\end{remark}
The next example shows that the exponent $2$ in \eqref{mainn} is optimal, that is, one can not substitute the $L^2$ norm of $f$ on the contact set with a $L^p$ norm for any $1\leq p<2$.
\begin{example}\label{optimalexponent} Let $\Omega$ be the unit ball in $\C$ and let $u(z)=\|z\|^{\alpha}-1\in PSH(\Omega)$, for $2>\alpha>0$. It is a matter of routine calculus to check that 
	$$(dd^c u)^n=f(z)=2^{n-1}n!\alpha^{n+1}\|z\|^{n\alpha-2n}.$$
	Switching to polar coordinates, one sees that $f\in L^p(\Omega)$, for any $1\leq p<\frac{2}{2-\alpha}$. On the other hand it can not be true, that
	
	\begin{equation*}
	\sup_{\Omega}u^{-}\leq \sup_{\pa \Omega} (u_{*})^{-}+Cdiam(\Omega)\Vert f\cdot\chi_{\{ \Gamma_{u_{*}}=\widetilde{u_{*}}\}}\Vert_{L^p(\Omega)}^{\frac{1}{n}},	\end{equation*}
	since if $\alpha\leq 1$ then $u(z)\geq \|z\|-1$ and hence $\{\Gamma_u=\tilde{u}\}$ consists of the sole origin. 
\end{example}
Let us remark that the problem of defining a correct notion of boundary values for $u\in PSH(\Omega)$ is a subtle one, as already noted in \cite{BT76}.
Interestingly, there  the authors remark that sufficiently general uniqueness
theorem for the Monge-Amp\`ere equation would imply nonexistence of nontrivial {\it inner} functions in the unit ball of $\mathbb C^n$, $n>1$. However, the existence of such functions
was later proven by Aleksandrov-\cite{A82} and Hakim-Sibony-L\o{}w-\cite{HS82,L82}. The nontrivial inner functions can in fact be used   to show 
that it is necessary to consider the lower semicontinuous regularization of plurisubharmonic functions on the boundary in our considerations:
\begin{example}\label{ex} Let $\Omega$ be the unit ball in $\mathbb C^n$, $n>1$, and let $F$ be a nontrivial holomorphic inner function on $\Omega$. As the radial limits of $F$ exist almost everywhere on $\pa \Omega$ and their absolute values are equal to $1$ again almost everywhere,
it is obvious that ${\displaystyle\limsup_{\Omega\ni z\rightarrow w\in\partial\Omega}|F(z)|^2-1=0}$ everywhere on $\pa \Omega$. It is well-known (see Theorem 19.1.3 in \cite{Rud08}) that in turn  ${\displaystyle \liminf_{\Omega\ni z\rightarrow\partial\Omega}|F(z)|^2-1=-1}$.
Consider  the maximal plurisubharmonic function $u(z):=|F(z)|^2-1$. Then $\sup_{\Omega}u^{-}=1$, the  last term in \eqref{mainn} vanishes, and it is obvious that Alexandrov-Bakelman-Pucci type estimate
is   possible only if $\liminf$-boundary values are taken into consideration.
\end{example}

\begin{remark}The same example shows the lack of uniqueness of the solution of the Dirichlet problem without boundary continuity: both $u_{1}(z)=|F(z)|^2-1$ and $u_{2}(z)=0$ are plurisubharmonic and satisfy ${\displaystyle\limsup_{\Omega\ni z\rightarrow\partial\Omega}u_{i}(z)=0}$, and $(dd^c u_{i})^n=0$. Also the global comparison principle fails in this generality.
\end{remark}

Our next example demonstrates the difference of taking the $\liminf$-boundary values only from {\it within} the considered domain as compared to the lower regularization, where approach from {\it outside} is allowed:
\begin{example}\label{ex3}
 Consider the function
$$u(z):=\sum_{n=3}^{\infty}a_n\log\left|z-\frac12-\frac1n\right|,$$
defined on the unit disc in $\mathbb C$,
where the constants $a_n>0$ are chosen so small that $u\left(\frac12\right)=\min_{\lbr|z|\leq\frac 12\rbr}u\geq-1$. Let $\lbrace\theta_j\rbrace_{j=1}^{\infty}$ be a dense sequence of angles in $[0,2\pi)$.
Define $$v(z):=\sum_{j=1}^{\infty}\frac1{2^n}u(e^{i\theta_j}z).$$
By construction $v|_{\lbrace|z|\leq\frac12\rbrace}\geq -1$, while $v_{*}|_{\partial\lbrace|z|\leq\frac12\rbrace}\equiv-\infty$.
Taking $\hat{v}(z):=e^{v(z)}$ results in a bounded subharmonic function (the boundedness from above  is clear) with constantly zero $\liminf$-boundary values. Finally note that

$$\hat{V}(z_1,z_2):=\hat{v}(z_1)$$
is a maximal plurisubharmonic function defined in, say, the unit ball in $\mathbb C^2$. Let the domain $\Omega$ be the ball centered at zero of radius $\frac12$.
Taking ${\displaystyle \liminf_{\Omega\ni z\rightarrow z_0\in\partial \Omega}\hat{V}(z)}$ rather than $\hat{V}_{*}(z_0)$ results  in a sharper Alexandrov-Bakelman-Pucci inequality.

\end{example}
Example \ref{ex3} also shows that is easy to produce discontinuous maximal plurisubharmonic functions. These are, however, not very useful in our considerations, since the Alexandrov-Bakelman-Pucci type inequality holds trivially for maximal plurisubharmonic functions.  In turn non-maximal discontinuous $PSH$ functions with
non-negative densities do not seem to be studied thoroughly in the literature. We believe that ABP type estimates in the discontinuous setting can be helpful in their study. But first of all one wants to know if such functions do exist. Hence we provide an example:

Let $K$ be a planar compact set, which is non-polar, contained in the imaginary axis $\{z: Re z=0\}$ and is irregular in the sense of potential theory (see \cite{Ran95} for these basic notions). An explicit construction is possible by choosing a sequence of intervals accumulating at $0$, with controlled lengths and suitably situated with respect to each other (see \cite{Sic97} for details). Irregularity can be established by using the Wiener criterion. Let $V_{K}^{*}$ be the extremal function associated to the set $K$ (or Green function for the complement of $K$ with pole at infinity). It is known that $V_{K}^{*}$ is positive and harmonic outside $K$, subharmonic in $\mathbb C$, and $\Delta V_{K}^{*}$ is a positive Borel measure supported on $K$. Because $K$ is irregular and non-polar, $V_{K}^{*}$ fails to be continuous.

\begin{example}\label{example} Let $\Omega\subset\subset\mathbb C^2$ be a bounded domain, contained in $\{(z,w)\in\mathbb C^2: |w|<1\}$, $K\times\{0\}\subseteq\Omega$ and $V_{K}^*$ is as above. Then 
	$$u(z,w):=V_{K}^{*}(z)+(Re z)^2(1+|w|^2)$$
is a bounded discontinuous plurisubharmonic function on $\Omega$ such that $\left(dd^c u\right)^2=f$, where $f\geq 0$ is not everywhere zero and is smooth.	
\end{example}
\begin{proof}
The discontinuity and boundedness are clear. Computing the complex Hessian, at a point outside $\{(z,w)\in\mathbb C^2: Re z=0\}$, that is, near which $u$ is $C^2$, gives one

$$\begin{pmatrix}u_{z\bar z}& u_{z\bar w}\\
u_{w\bar z}& u_{w\bar w}\end{pmatrix}=\begin{pmatrix}\frac{1}{4}\Delta V_{K}^{*}+\frac{1}{2}(1+|w|^2)& w Re z\\
 \bar w Re z& (Re z)^2\end{pmatrix},$$
so the determinant is $\frac{1}{2}(Re z)^2(1-|w|^2)$, which extends to a non-negative and smooth function on $\bar \Omega$. 
On the other hand $\left(dd^c u\right)^2$ can put no mass on $\{(z,w)\in\mathbb C^2: Re z=0\}$
since $\frac{1}{4}\Delta V_{K}^{*}$ is killed by the term $(Re z)^2$.
\end{proof}

Theorem \ref{ABPestimate} has several immediate corollaries.

\begin{corollary}\label{pota}Let $u$ be as in Theorem \ref{ABPestimate}. Suppose moreover that $f\in L^{\infty}(\Omega)$. Then for any relatively compact  subdomain $U\subseteq\Om$ the following estimate holds:
	$$\sup_{U}u^{-}\leq \sup_{\pa U} (u_{*})^{-}+Cdiam(U)Vol(U)^{1/(2n)}\Vert f\Vert_{L^{\infty}(U)}^{1/n},$$
	where $C$ is a numerical constant dependent only on the dimension $n$.
\end{corollary}

\begin{proof}
This follows trivially by estimating the last term in \eqref{mainn} by  $Vol(U)^{1/(2n)}\Vert f\Vert_{L^{\infty}(U)}^{1/n}$.
\end{proof}

\begin{corollary}Under the  assumptions of Theorem \ref{ABPestimate}, and if the supremum of $u^-$ is not attained on the boundary, then the contact set $\{ \Gamma_{u_{*}}=\widetilde{\, u_{*}}\}$ of the function $u$ has positive Lebesgue measure. In a sense such plurisubharmonic functions have ''pointwise convex'' lower semicontinuous regularizations (their graphs allow supporting real hyperplanes) on a big set. 
\end{corollary}

 Concerning Wang's problem we observe the following.
 \begin{example}\label{wanggg}The $L^2$ norm over the contact set can not control the $L^{p}$ norm over the whole domain and vice versa, which is demonstrated by the the examples below (in each case we put $(dd^c u)^n=f$ and   $(dd^c u_{\varepsilon})^n=f_{\varepsilon}$ ): \begin{enumerate}
 		\item	Let $u$ be a function of the type considered  in Remark \ref{optimalconstant}, namely $u(z)=-\sqrt{d^2-\|z\|^2}$. We have $f\not\in L^1(B_{d}(0))$, so we perturb $u$ near the boundary of the ball by setting $u_{\varepsilon}=u$ on  $B_{d-2\varepsilon}(0)$, $u_{\varepsilon}=\frac{\sqrt{d^2-(d-\varepsilon)^2}}{\varepsilon}(\|z\|-d)$   on $B_{d}(0)\setminus \overline{B_{d-\varepsilon}(0)}$ and $u_{\varepsilon}$ is extended by using a smooth transition function on $\overline{B_{d-\varepsilon}(0)\setminus B_{d-2\varepsilon}(0)}$, keeping convexity.  The contact set of $u_{\varepsilon}$ is $\overline{B_{\frac{d}{2}}(0)}$.
 		
 		Then  $\| f_{\varepsilon}\chi_{\{\Gamma_{ {u_{\varepsilon}}_{*}}=\widetilde{{u_{\varepsilon}}_{*}}\}}\|_{L^2(\Omega)}$ stays fixed, whereas $\| f_{{\varepsilon}} \|_{L^{p}(\Omega)}>\| f_{{\varepsilon}} \|_{L^{p}(B_{d-2\varepsilon}(0))}$ is arbitrarily big for any $p\geq1$.
 		
 		\item	Let $u$ be the function form Example \ref{optimalexponent} namely $u(z)=\|z\|^{\alpha}-1, 0<\alpha<1$. We construct $u_{\varepsilon}$ by first properly normalizing $u$, that is putting $\varepsilon\left(\left(\frac{\|z\|}{d}\right)^{\alpha}-1\right)$ and after that  truncating the ''tip'' of its graph in a small ball. This is done by setting $u_{\varepsilon}=-\sqrt{\varepsilon^2-\|z\|^2}$ in a neighborhood of $\overline{B_{\frac{\varepsilon^2}{2d}}(0)}$ and patching this function smoothly with the normalized $u$ over $B_{\frac{\varepsilon^2}{d}}(0)\setminus B_{\frac{\varepsilon^2}{2d}}(0)$, while keeping the plurisubharmonicity.  This is again done 
 		by using a smooth transition function. The patching is possible because $u_{\varepsilon}\left(\frac{\varepsilon^2}{d}\right)=\varepsilon\left(\left(\frac{\varepsilon^2}{d^2}\right)^{\alpha}-1\right)>-\sqrt{\varepsilon^2-\left(\frac{\varepsilon^2}{2d}\right)^2}=u_{\varepsilon}\left(\frac{\varepsilon^2}{2d}\right)$ if $\varepsilon$ is small enough. As in Remark \ref{optimalconstant}, the contact set is $\overline{B_{\frac{\varepsilon^2}{2d}}(0)}$ and
 		$$\| f_{\varepsilon}\chi_{\{\Gamma_{ {u_{\varepsilon}}_{*}}=\widetilde{{u_{\varepsilon}}_{*}}\}}\|_{L^2(\Omega)}\sim C\left(\frac{\varepsilon}{\sqrt{4d^2-\varepsilon^2}}\right)^n\to 0,$$ 
 		whereas, using Example \ref{optimalexponent},
 		
 		$$\| f_{\varepsilon}\|_{L^p(\Omega)}>\| f_{\varepsilon}\|_{L^p\left(B_{d}(0)\setminus B_{\frac{\varepsilon^2}{d}}(0)\right)}\sim C_{1}\varepsilon^n+ C_2(\varepsilon^{\frac{pn+2(n\alpha-2n)p+4n}{p}}).$$
 		Hence $\| f_{\varepsilon} \|_{L^{p}(\Omega)}$ is either bounded and separated from zero if $1\leq p<\frac{4}{3-2\alpha}$ or arbitrarily big if $ p\geq\frac{4}{3-2\alpha}$.
 		
 		\item	Let $u_{\varepsilon}$ be defined as follows. Let $u_{\varepsilon}=\log\frac{\|z\|}{d}$ on $B_{d}(0)\setminus B_{\frac{3\varepsilon}{d}}(0)$, $u_{\varepsilon}=\frac{\left|\log\frac{\varepsilon}{d}\right|}{\varepsilon}\|z\|^2-\left|\log\frac{\varepsilon}{d}\right|$ on $ B_{\frac{\varepsilon}{3d}}(0)$. Note that $u_{\varepsilon}(w)>u_{\varepsilon}(v)$ if $\|w\|=\frac{3\varepsilon}{d}$, $\|v\|=\frac{\varepsilon}{3d}$ and $\varepsilon$ is small enough. Now we extend $u_{\varepsilon}$ on $B_{\frac{3\varepsilon}{d}}(0)\setminus B_{\frac{\varepsilon}{3d}}(0)$ in such a way that $u_{\varepsilon}$ is increasing with $\|z\|$, smooth, plurisubharmonic and $f_{\varepsilon}$ is decreasing. The contact set is the closed ball of radius $2d-\sqrt{4d^2-\varepsilon}\gtrsim\frac{\varepsilon}{4d}$.
 		Now $$\| f_{\varepsilon}\chi_{\{\Gamma_{ {u_{\varepsilon}}_{*}}=\widetilde{{u_{\varepsilon}}_{*}}\}}\|_{L^2(\Omega)}>\| f_{\varepsilon}\|_{L^2\left(B_{\frac{\varepsilon}{4d}}(0)\right)}\sim C\left|\log\frac{\varepsilon}{d}\right|^{n}\to\infty,$$
 		whereas
 		$$\| f_{\varepsilon}\|_{L^p(\Omega)}=\| f_{\varepsilon}\|_{L^p\left(B_{\frac{3\varepsilon}{d}}(0)\right)}\lesssim C\varepsilon^{\frac{n(2-p)}{p}}\left|\log\frac{\varepsilon}{d}\right|^{n}\to0,$$
 		for any $1\leq p<2$.
 	\end{enumerate}
 \end{example}
 \begin{remark} Following the proof of Theorem \ref{ABPestimate}, one sees that the assumption $f\in L^{2}(\Omega)$ can be changed to just $ f\in L^{2}(\lbrace \Gamma_{u_*}=\widetilde{\, u_*}\rbrace)$ and the first of the Examples \ref{wanggg} shows that under such assumption the ABP estimate applies to a wider range of plurisubharmonic functions.
 \end{remark}

 \medskip
 {\bf Acknowledgments}
 
 The first named author was partially supported by National Science Centre, Poland, grant no. 2017/26/E/ST1/00955. The second named author was partially supported by  the  National Science Centre, Poland grant no 2013/08/A/ST1/00312.


\begin{thebibliography}{KKPS01}
\bibitem[Ale82]{A82} A.B. Aleksandrov:
The existence of inner functions in a ball (Russian),
Mat. Sb. (N.S.) {\bf118(160)} (1982), no. 2, 147-163.
	
	 \bibitem[AIM06]{AIM06} K. Astala, T. Iwaniec and G. Martin:
	 Pucci's conjecture and the Alexandrov inequality for
	 	elliptic PDEs in the plane, J. Reine Angew. Math.
 {\bf 591} (2006), 49-74.
	\bibitem[BT76]{BT76} E. Bedford and B. Taylor:
	The Dirichlet problem for a complex Monge-Amp\`ere
		equation,
Invent. Math. {\bf 37}, 1976,
	no. 1, 1-44.
	\bibitem[BT82]{BT82} E. Bedford and B. Taylor:
	A new capacity for plurisubharmonic functions,
Acta Math. {\bf149} (1982), no. 1-2, 1-40. 
\bibitem[Blo96]{Blo96} Z. B\l ocki: The complex Monge-Amp\'ere operator in hyperconvex domains, Ann. Scuola Norm. Sup. Pisa Cl. Sci. {\bf 23} (1996), no. 4, 721-747.	
\bibitem[Blo04]{Blo04}	Z. B\l ocki: On the definition of the Monge-Amp\`{e}re operator in $\mathbb C^2$, Math. Ann. {\bf 328} (2004), 415-423.
\bibitem[Blo05]{Blo05} Z. B\l ocki:  On uniform estimate in Calabi-Yau theorem, Proceedings of SCV2004, Beijing, Sci. China Ser. A Math. {\bf48} Supp. (2005), 244-247.
\bibitem[Blo06]{Blo06} Z. B\l ocki: The domain of definition of the Complex Monge-Amp\`{e}re operator, Amer. J. Math. {\bf 128} (2006), 519-530.

\bibitem[Cab95]{Cab95} X. Cabr\'{e}:
On the Alexandroff-Bakel' man-Pucci estimate and the
	reversed H\"{o}lder inequality for solutions of elliptic and
	parabolic equations, Comm. Pure Appl. Math. {\bf 48} (1995), no.5, 539-570.

\bibitem[CC95]{CC95} L. Caffarelli and X. Cabr\'{e}:
Fully nonlinear elliptic equations, American Mathematical Society Colloquium Publications, {\bf 43},
American Mathematical Society, Providence, RI, 1995, ISBN: 0-8218-0437-5,
vi+104.

\bibitem[CCKS96]{CCKS96} L. Caffarelli, M. Crandall, M. Kocan and A. Swi\k{e}ch:
On viscosity solutions of fully nonlinear equations with
	measurable ingredients,
Comm. Pure Appl. Math.
 {\bf 49},
(1996), no. 4, 365-397.

\bibitem[CP92]{CP92} U. Cegrell and L. Persson:
The Dirichlet problem for the complex Monge-Amp\`ere
	operator: stability in {$L^2$}, Michigan Math. J.  {\bf 39}, (1992), no. 1, 145-151.

\bibitem[CIL92]{CIL92} M. Crandall, H. Ishii and P. Lions: User's guide to viscosity solutions of second order partial
	differential equations, Bull. Amer. Math. Soc. (N.S.),
 {\bf 27} (1992), no. 1, 1-67.
 
 \bibitem[DD19]{DD18} S. Dinew and \.Z. Dinew: Differential Tests for Plurisubharmonic Functions and Koch Curves, to appear in Potential Anal (2018). https://doi.org/10.1007/s11118-018-9686-6.
 \bibitem[EGZ11]{EGZ11}  P. Eyssidieux, V. Guedj and A. Zeriahi: Viscosity solutions to degenerate complex
 Monge-Amp\`ere equations, Comm. Pure Appl. Math. {\bf 64} (2011), no. 8, 1059-1094.

\bibitem[GT01]{GT01} D. Gilbarg and N. Trudinger: Elliptic Partial Differential Equations of Second Order, Reprint of the 1998 edition,  Classics in Mathematics,  Springer-Verlag, Berlin, 2001, ISBN: 3-540-41160-7, xiv+517.
\bibitem[GLZ]{GLZ} V. Guedj, H.C. Lu and A. Zeriahi: Weak subsolutions to complex Monge-Amp\`ere equations, preprint  arXiv:1703.06728.
\bibitem[Gut01]{Gu01} C. Guti\'errez: The Monge-Amp\`ere Equation, Progress in Nonlinear Differential Equations
and their Applications, {\bf44}, Birkh\"auser Boston, Inc., Boston, MA, 2001, ISBN: 0-8176-4177-7, xii+127.
\bibitem[HS82]{HS82} M. Hakim and N. Sibony:
Fonctions holomorphes born\'ees sur la boule unit\'e de $\mathbb C^n$ (French), 
Invent. Math. {\bf 67} (1982), no. 2, 213-222.
\bibitem[Kol05]{Kol05} S. Ko\l odziej: The complex Monge-Amp\`ere equation and pluripotential	theory,
Mem. Amer. Math. Soc. {\bf 178} (2005), no. 840, x+64.
\bibitem[L\o{}w82]{L82} E.L\o{}w: A construction of inner functions on the unit ball in $\mathbb C^p$,
Invent. Math. {\bf 67} (1982), no. 2, 223-229. 

\bibitem[Ran95]{Ran95} T. Ransford: Potential Theory in the Complex Plane, London Mathematical Society Student Texts, {\bf 28}, Cambridge University Press, Cambridge, 1995, ISBN 0-521-46120-0; 0-521-46654-7, x+232.
 \bibitem[RT77]{RT77} J. Rauch and  B.A. Taylor:
The Dirichlet problem for the multidimensional Monge-Amp\`ere equation,
Rocky Mountain J. Math. {\bf 7} (1977), no. 2, 345-364.
\bibitem[Rud08]{Rud08} W. Rudin: Function Theory in the Unit Ball of {$\Bbb C^n$},
Classics in Mathematics,
Reprint of the 1980 edition,
Springer-Verlag, Berlin, 2008,
 ISBN 978-3-540-68272-1,
xiv+436.
\bibitem[Sic81]{Sic81} J. Siciak: Extremal plurisubharmonic functions in ${\mathbb C}^{n}$,	Ann. Polon. Math. {\bf 39}
	(1981), 175-211.
	




 \bibitem[Sic97]{Sic97}  J. Siciak:
 Wiener's type regularity criteria on the complex plane,
 Volume dedicated to the memory of W\l odzimierz Mlak, Ann. Polon. Math.
 {\bf 66},
 (1997),
 203-221.

\bibitem[Wan12]{Wan12} Y. Wang: A viscosity approach to the Dirichlet problem for complex Monge-Amp\`{e}re equations, Math. Z. {\bf 272} (2012), no. 1-2, 497-513. 
 
\bibitem[Zer13]{Zer13} A. Zeriahi: A viscosity approach to degenerate complex Monge-Amp\`{e}re equations, Ann. Fac. Sci. Toulouse Math. (6) {\bf 22} (2013),
no. 4, 843-913.
\end{thebibliography}
\end{document}